\documentclass[reqno,12pt]{amsart}
\usepackage{graphicx}
\usepackage{epsfig}
\usepackage{amssymb,amsthm,amsmath,amsfonts,amstext,mathrsfs,fancybox}
\usepackage{enumerate}
\usepackage{latexsym}
\usepackage{url}
\usepackage{fancyhdr}
\usepackage{color}
\usepackage[unicode]{hyperref}

\hypersetup{colorlinks=true,
linkcolor=blue,
citecolor=blue,
urlcolor=blue,
filecolor=blue,
bookmarksnumbered=true,
pdfstartview=FitH,
pdfhighlight=/N}

\newtheorem{thm}{Theorem}

\newtheorem{prop}{Proposition}

\newtheorem{defin}{Definition}

\input xy
\xyoption{all}

\def\d{\,{\rm{d}}}

\def\kk{\varkappa}
\def\E{\mathfrak{E}}
\def\f{\mathscr{F}}

\def\m{\,{\mathfrak{m}}}

\def\q{\,{\mathbf{q}}}

\title[The Minkowski question mark function]
{ The Minkowski $?(x)$ function, a class of singular measures, quasi-modular and mean-modular forms. I}
\author[G. Alkauskas]{Giedrius Alkauskas}
\address{Vilnius University, Department of Mathematics and Informatics, Naugarduko 24, LT-03225 Vilnius, Lithuania}
\email{giedrius.alkauskas@mif.vu.lt}

\begin{document}
\begin{abstract}
The Minkowski question mark function is a rich object which can be explored from the perspective of dynamical systems, complex dynamics, metric number theory, multifractal analysis, transfer operators, integral transforms, and as a function itself via analysis of continued fractions and convergents. Our permanent target, however, was to get arithmetic interpretation of the moments of $?(x)$ (which are relatives of periods of Maass wave forms) and to relate the function $?(x)$ to certain modular objects. In this paper we establish this link, embedding $?(x)$ not into the modular-world itself, but into a space of functions which are generalizations and which we call \emph{mean-modular forms}. For this purpose we construct a wide class of measures, and also investigate modular forms for congruence subgroups which additionally satisfy the three term functional equation. From this perspective, the modular forms for the whole modular group as well as the Stieltjes transform of $?(x)$ (the dyadic period function) minus the Eisenstein series of weight $2$ fall under the same uniform definition. The main result is the construction of the canonical isomorphism between the spaces of quasi-modular forms and mean-modular forms. This gives unexpected Minkowski question mark function-related interpretation of quasi-modular forms.
\end{abstract}
\pagestyle{fancy}
\fancyhead{}
\fancyhead[LE]{{\sc The Minkowski $?(x)$ function}}
\fancyhead[RO]{{\sc G. Alkauskas}}
\fancyhead[CE,CO]{\thepage}
\fancyfoot{}
\date{July 2, 2015}
\thanks{The research of the author was supported by the Research Council of Lithuania grant No. MIP-072/2015. The author sincerely acknowledges hospitality of Prof. J. Steuding in W\"{u}rzburg University, Prof. D. Mayer in Clausthal Technical University, Max Planck Institute for Mathematics in Bonn, and also thanks Prof. N. Diamantis and Prof. K. Bringmann for showing interest in these results.}

\maketitle

\section{Introduction}
The relation between continued fractions and modular functions is and old and deep subject; see, for example, \cite{cais,duke, manin}. In this paper we provide yet another example of this relation of a very different sort.\\
  
The Minkowski question mark function $?(x):[0,1]\mapsto[0,1]$ is defined by
\begin{eqnarray*}
?([0,a_{1},a_{2},a_{3},\ldots])=2\sum\limits_{\ell=1}^{\infty}(-1)^{\ell+1}2^{-\sum_{j=1}^{\ell}a_{j}},\quad a_{j}\in\mathbb{N};
\end{eqnarray*}
$x=[0,a_{1},a_{2},a_{3},\ldots]$ stands for a representation of $x$ by a regular continued fraction. In view of the current paper, note that the Minkowski question mark function can be defined also in terms of semi-regular continued fractions. These are given by 
\begin{eqnarray*}
[[b_{1},b_{2},b_{3},\ldots]]=\cfrac{1}{b_{1}-\cfrac{1}{b_{2}-\cfrac{1}{b_{3}-\ddots}}},
\end{eqnarray*}
where integers $b_{i}\geq 2$. Each real irrational number $x\in(0,1)$ has a unique representation in this form, and rationals $x\in(0,1)$ have two representations: one finite and one infinite which ends in $[[2,2,2,\ldots]]=1$. It was proved in \cite{alkauskas3} that
\begin{eqnarray*}
?([[b_{1},b_{2},b_{3},\ldots]])=
\sum\limits_{\ell=1}^{\infty}2^{\ell-\sum_{j=1}^{\ell}b_{j}}.
\end{eqnarray*}
The function $?(x)$ is continuous, strictly increasing, and singular. For $x\in[0,1]$, it satisfies functional equations
\begin{eqnarray*}
?(x)=\left\{\begin{array}{c@{\qquad}l} 1-?(1-x),
\\ 2?\big{(}\frac{x}{x+1}\big{)}. \end{array}\right.
\end{eqnarray*}
These equations are responsible for the rich arithmetic nature of $?(x)$ and its relations (at least analogies) to the objects in the modular-world \cite{alkauskas1,alkauskasi}: for example, if we define 
\begin{eqnarray*}
G(z)=\int\limits_{0}^{1}\frac{x}{1-xz}\d ?(x),
\end{eqnarray*}
then $G(z)=o(1)$ if $z\rightarrow\infty$ and the distance to $\mathbb{R}_{+}$ remains bounded away from $0$, and
\begin{eqnarray*}
\frac{1}{z}+\frac{1}{z^{2}}G\Big{(}\frac{1}{z}\Big{)}+2G(z+1)=G(z),\quad 
z\in\mathbb{C}\setminus[1,\infty).
\end{eqnarray*}
In this paper we exhibit explicitly the connection of $?(x)$ to the modular world. The factor $``2"$ in the above formula - an intrinsic constant which comes from the dyadic nature of $?(x)$ - was always an obstacle which prevented an application of many techniques (Hecke operators, modularity, Fourier series) to the theory of $?(x)$. Now it appears that there exists a natural way to integrate $?(x)$ into the modular world, and this factor $``2"$ is no longer an obstacle but rather the reason why this integration is possible. For this purpose, first, we construct a wide generalization of $?(x)$.
\section{A class of functions}
Here we present a new way to construct a wide class of continuous fractal functions which encode the self-similarity via semi-regular continued fractions.
\begin{prop}\label{prop1}
Let $\mathbf{q}=\{q_{\ell}:2\leq \ell<\infty\}$ be the sequence of complex numbers such that
\begin{eqnarray*}
\sum\limits_{\ell=2}^{\infty}q_{\ell}=1,\quad
\sum\limits_{\ell=2}^{\infty}|q_{\ell}|<+\infty,\quad\sup\limits_{\ell}|q_{\ell}|<1.
\end{eqnarray*}
Then there exists the function $\mu=\mu_{\q}:[0,1]\mapsto\mathbb{C}$ with the following properties.
\begin{itemize}
\item[1)] It is continuous, $\mu(0)=0$, $\mu(1)=1$.
\item[2)] The function $\mu$ is of bounded variation. If all $q_{\ell}$ are real and non-negative, then $\mu$ is non-decreasing; if all $q_{\ell}$ are strictly positive, then $\mu$ is strictly increasing.
\item[3)] The function $\mu$ has the following self-similarity property:
\begin{eqnarray*}
\mu\Big{(}\frac{1}{\ell-x}\Big{)}=q_{\ell}\cdot\mu(x)+\sum\limits_{j=\ell+1}^{\infty}q_{j},\quad 2\leq \ell<\infty,\quad x\in[0,1].
\end{eqnarray*}
\item[4)] if $q_{\ell}=2^{1-\ell}$, $\ell\geq 2$, then $\mu(x)=?(x)$.
\end{itemize}
\end{prop}
\begin{proof}To construct such a function, we use iterations. As an initial state, set $\mu_{0}(x)=x$, $x\in[0,1]$. Then define $\mu_{w+1}$ piecewise recurrently by
\begin{eqnarray*}
\mu_{w+1}(x)=q_{\ell}\cdot\mu_{w}\Big{(}\ell-\frac{1}{x}\Big{)}+\sum\limits_{j=\ell+1}^{\infty}q_{j},\quad
x\in\Big{[}\frac{1}{\ell},\frac{1}{\ell-1}\Big{]},\quad w\geq 0.
\end{eqnarray*}
By induction we see that $\mu_{w+1}(0)=0$, $\mu_{w+1}(1)=1$, and that $\mu_{w}$ is continuous. Now, consider the following series
\begin{eqnarray}
\mu_{0}(x)+\sum\limits_{w=0}^{\infty}\big{(}\mu_{w+1}(x)-\mu_{w}(x)\big{)}.
\label{serr}
\end{eqnarray}
Let $\sup_{\ell}|q_{\ell}|=\delta<1$, and $\sup_{[0,1]}|\mu_{1}(x)-\mu_{0}(x)|=M$. By the very construction,
\begin{eqnarray*}
\mu_{w+1}(x)-\mu_{w}(x)=q_{\ell}\cdot\Big{(}\mu_{w}\Big{(}\ell-\frac{1}{x}\Big{)}-\mu_{w-1}\Big{(}\ell-\frac{1}{x}\Big{)}\Big{)},
\quad x\in\Big{[}\frac{1}{\ell},\frac{1}{\ell-1}\Big{]},\quad w\geq 1.
\end{eqnarray*}
So, for $w\geq 1$,
\begin{eqnarray*}
\sup\limits_{x\in[0,1]}|\mu_{w+1}(x)-\mu_{w}(x)|\leq \delta\cdot\sup\limits_{x\in[0,1]}|\mu_{w}(x)-\mu_{w-1}(x)|.
\end{eqnarray*}
Thus, the series (\ref{serr}) is majorized by the series $\sum_{w} M\delta^{w}$, and so the function
\begin{eqnarray*}
\mu(x)=\lim\limits_{w\rightarrow\infty}\mu_{w}(x)
\end{eqnarray*}
is continuous and satisfies all of the needed properties, as can be checked.
\end{proof}
We call this function $\mu_{\q}$ the \emph{$\q-$question mark function}. For example, the Figures 1,2,3 shows the graph of these in cases $\q=(\frac{2}{3},\frac{1}{3},0,0,\ldots)$, $\q=(\frac{4}{7},\frac{2}{7},\frac{1}{7},0,0,\ldots)$, $\q=(\frac{4}{7},\frac{4}{7},-\frac{1}{7},0,0,\ldots).$\\

\begin{figure}[h]
\begin{center}
\epsfig{file=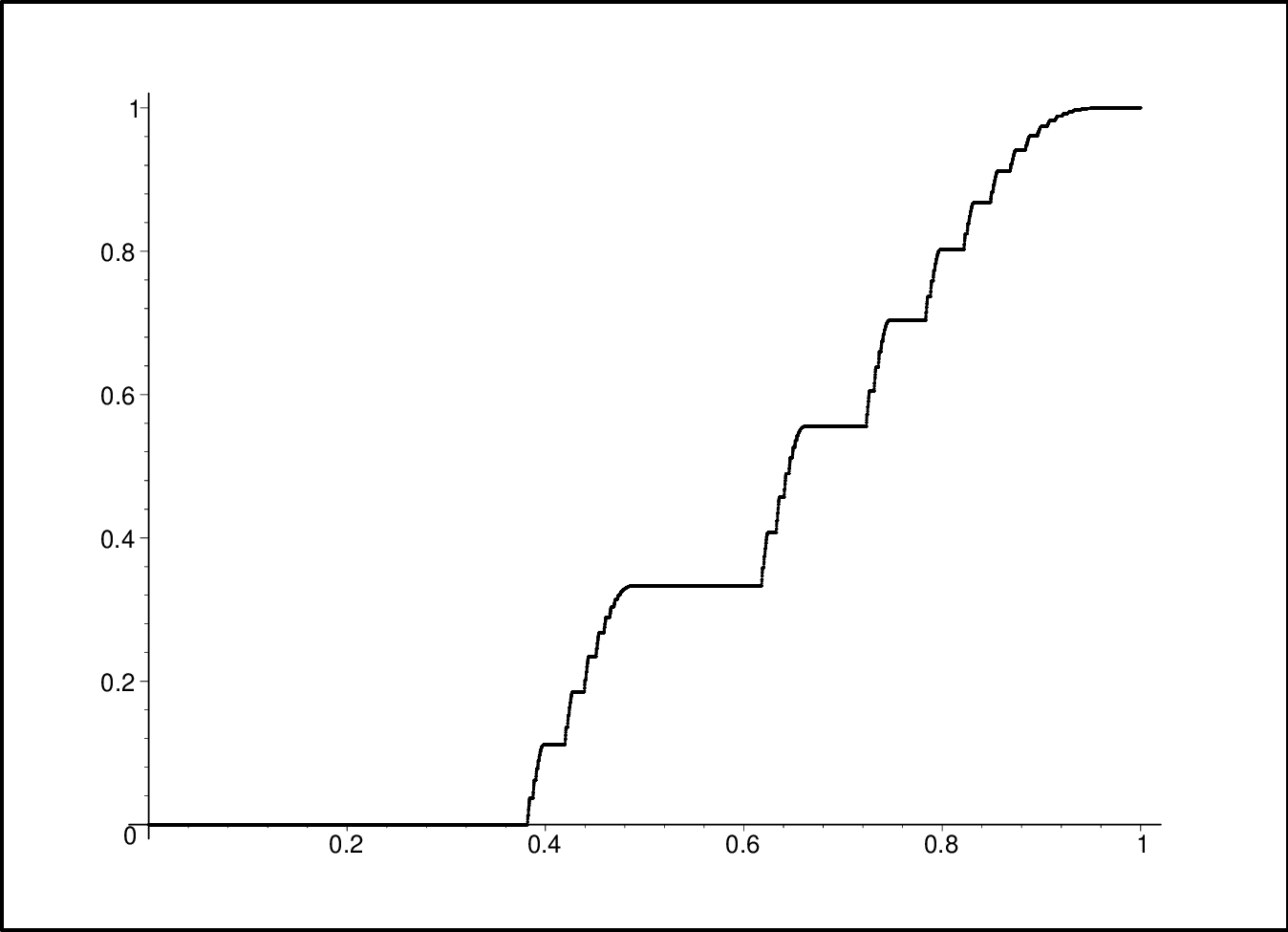,width=240pt,height=340pt,angle=-90}
\caption{$(\frac{2}{3},\frac{1}{3})$-question mark function, $x\in[0,1]$}
\end{center}
\label{figure1}
\end{figure}

\begin{figure}[h]
\begin{center}
\epsfig{file=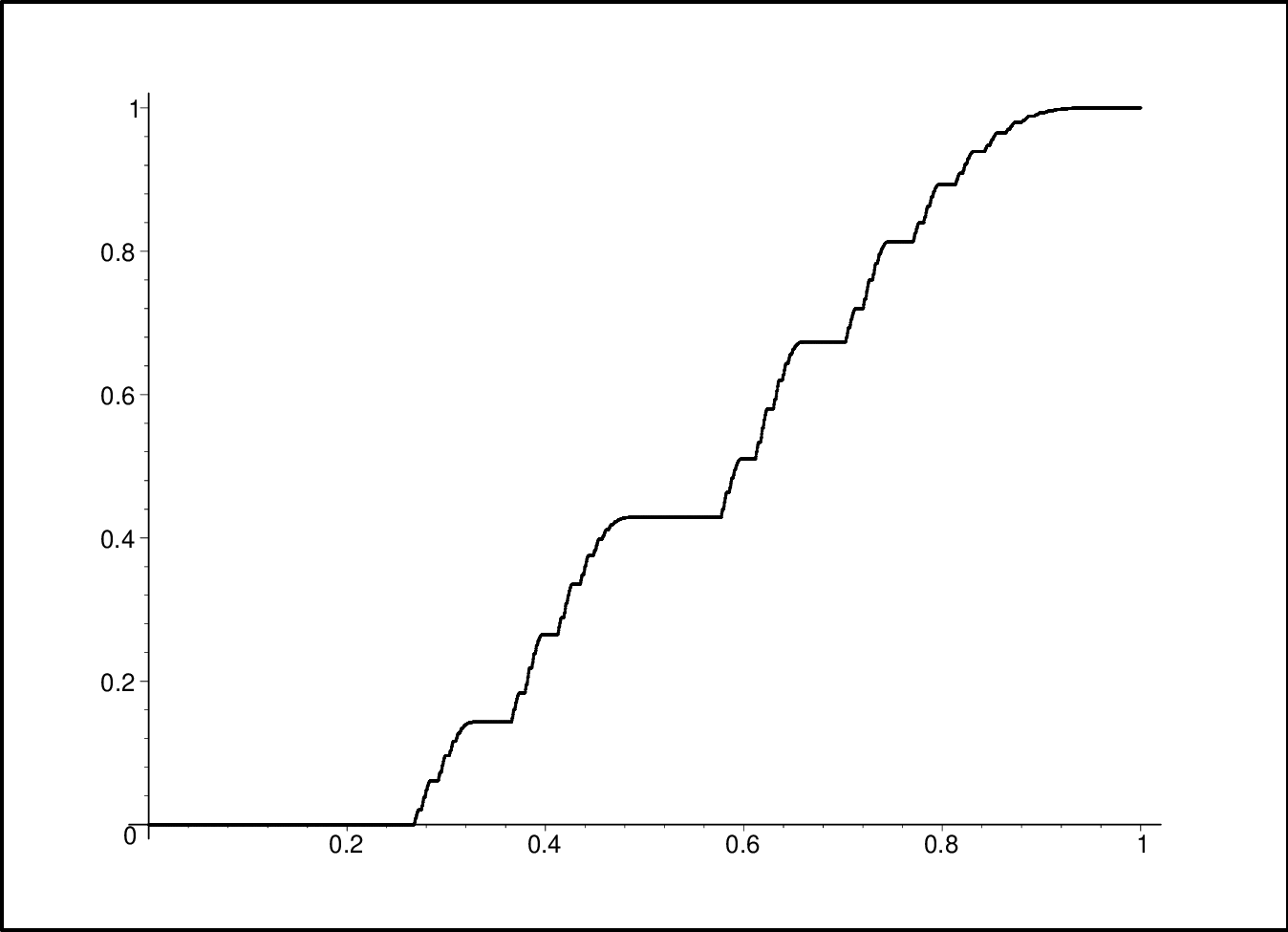,width=240pt,height=340pt,angle=-90}
\caption{$(\frac{4}{7},\frac{2}{7},\frac{1}{7})$-question mark function, $x\in[0,1]$}
\end{center}
\label{figure2}
\end{figure}

\begin{figure}[h]
\begin{center}
\epsfig{file=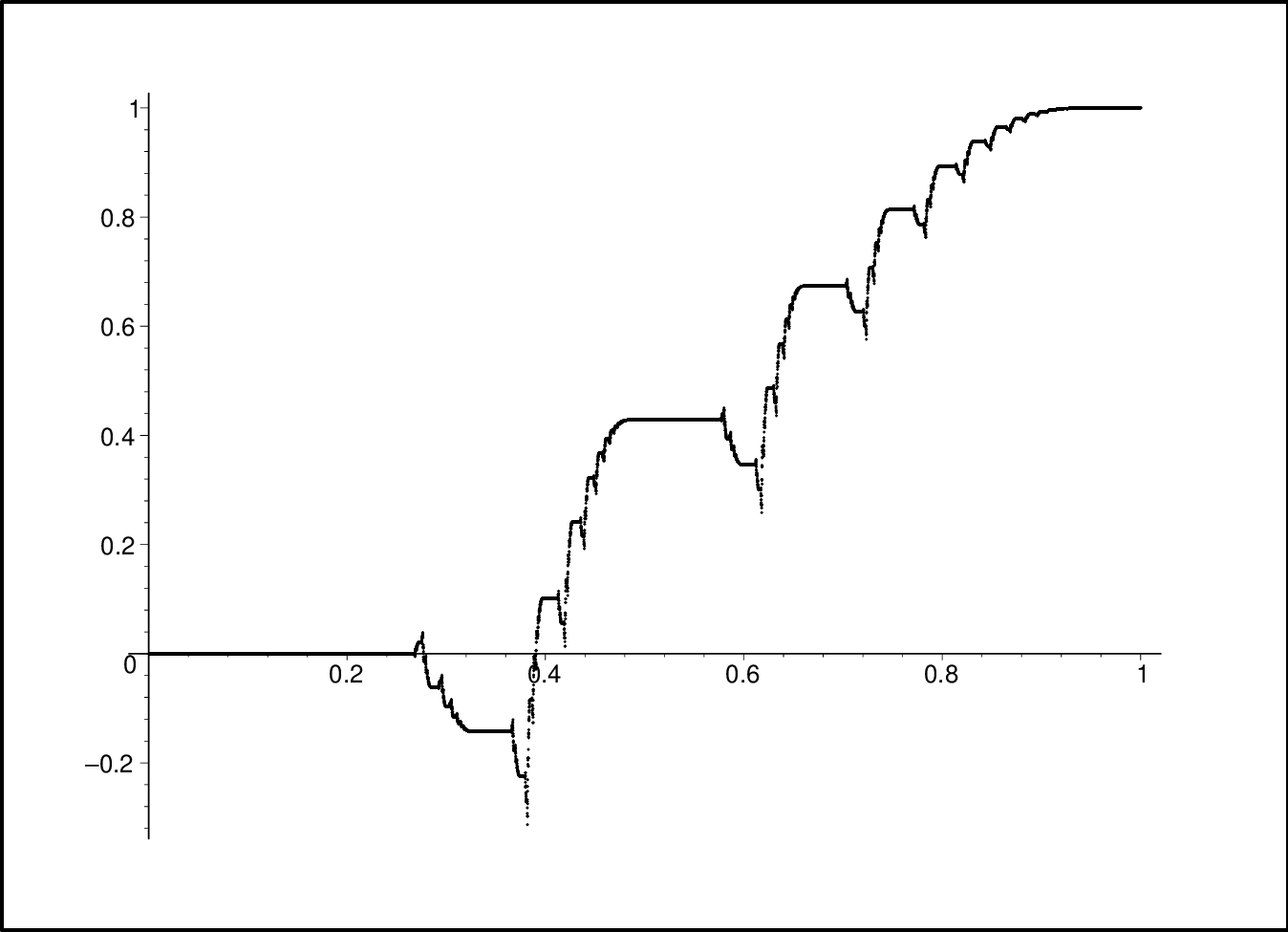,width=240pt,height=340pt,angle=-90}
\caption{$(\frac{4}{7},\frac{4}{7},-\frac{1}{7})$-question mark function, $x\in[0,1]$}
\end{center}
\label{figure3}
\end{figure}

\begin{figure}[h]
\begin{center}
\epsfig{file=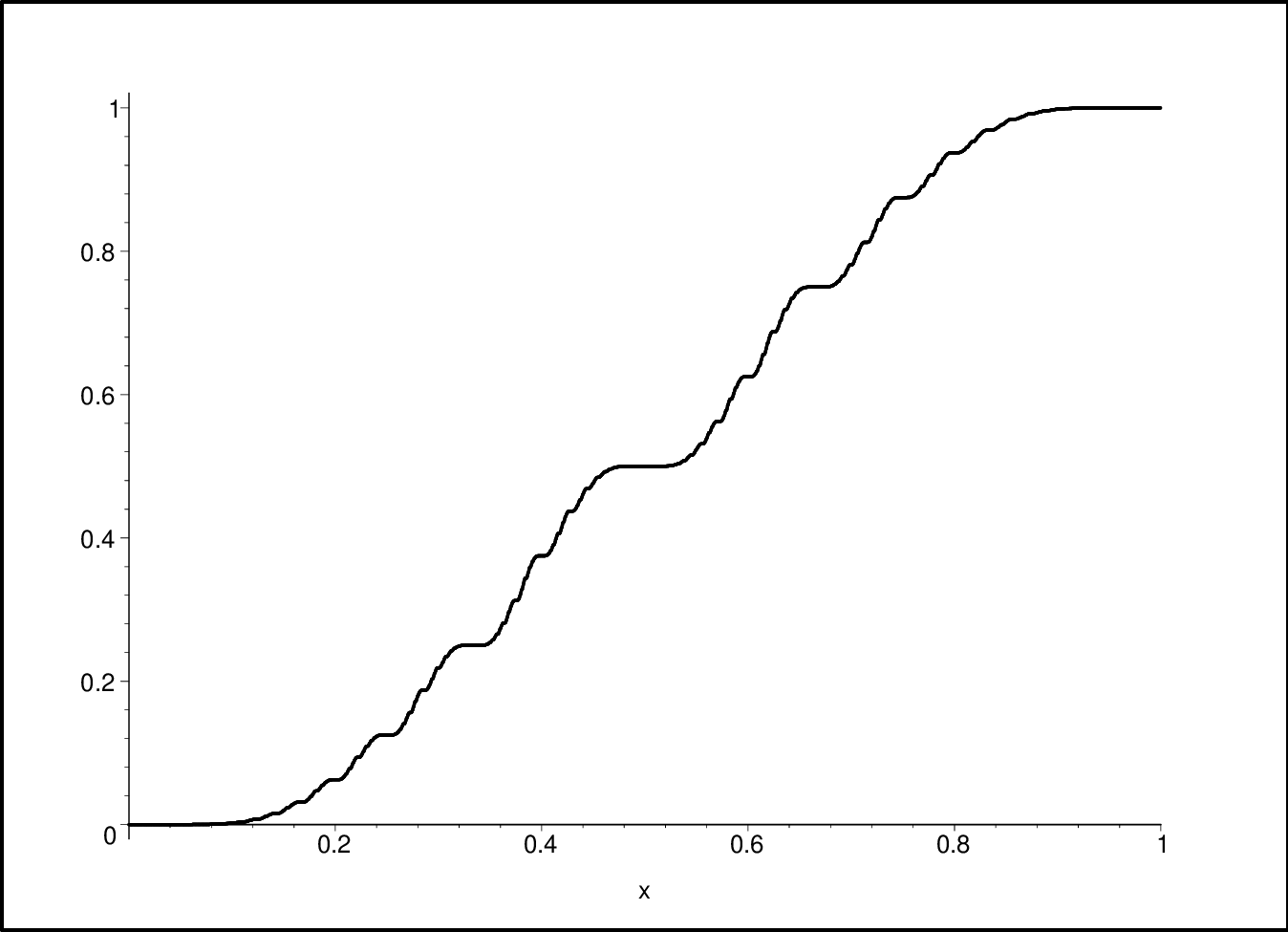,width=240pt,height=340pt,angle=-90}
\caption{The Minkowski question mark function, $x\in[0,1]$}
\end{center}
\label{figure4}
\end{figure}

As an aside, let us define
\begin{eqnarray*}
\m_{\q}(s)=\int\limits_{0}^{1}e^{xs}\d \mu_{\q}(x),\quad 
p_{\q}(s)=\sum\limits_{\ell=2}^{\infty}q_{\ell}e^{-i\ell s}.
\end{eqnarray*}
It is unknown whether $\m_{\q}(is)$ vanishes at infinity for $s\in\mathbb{R}$ in case of the Minkowski question mark function - this is \emph{the Salem's problem} \cite{alkauskas4,alkauskas5}. Most likely, all $\m_{\q}(is)$ vanish at infinity. It is out of the scope of the current paper, but we mention that the integral equation for the Laplace-Stieltjes transform of 
$?(x)$, defined by \cite{alkauskas1}
\begin{eqnarray*}
\m(s)=\int\limits_{0}^{1}e^{xs}\d ?(x),\quad s\in\mathbb{C},
\end{eqnarray*}
is compatible with this much more general construction. So, the function $\m_{\q}(s)$ is entire, and it satisfies the following integral equation
\begin{eqnarray*}
i\m_{\q}(is)p_{\q}(s)=\int\limits_{0}^{\infty}\m'_{\q}(it)J_{0}(2\sqrt{st})\d t,\quad s>0;
\end{eqnarray*}
the integral converges conditionally. On the other hand, the three term functional equation for the Stieltjes transform of $?(x)$ is compatible only with a narrow one parameter subclass of such $\q$'s which we introduce now, since this is our main object.
\section{A special subclass}\indent We will now focus on the important sequence $\q$ given by $q_{\ell}=(1-\kk)\kk^{\ell-2}$, $\ell\geq 2$, $\kk\in\mathbb{C}$, $|\kk|<1$, $|1-\kk|<1$. Let therefore $\mu_{\q}=\mu_{\kk}$ in this case.
 Let us define
\begin{eqnarray}
G(\kk,z)=\int\limits_{0}^{1}\frac{1}{\frac{1}{x}-z}\d\mu_{\kk}(x),\quad z\in\mathbb{C}\setminus[1,\infty).
\label{g-defin}
\end{eqnarray}
Note that this implies 
\begin{eqnarray*}
\frac{\partial^{s}}{\partial z^{s}}\,G(\kk,z)=s!\int\limits_{0}^{1}\frac{1}{(\frac{1}{x}-z)^{s+1}}\d\mu_{\kk}(x).
\end{eqnarray*}
\begin{prop}
\label{prop2}
The function $G(\kk,z)$ satisfies the three term functional equation
\begin{eqnarray}
G(\kk,z+1)-\kk G(\kk,z)=\frac{(1-\kk)}{(1-z)^2}G\Big{(}\kk,\frac{1}{1-z}\Big{)}+\frac{1-\kk}{1-z},\quad z\in\mathbb{C}\setminus[1,\infty).
\label{centr}
\end{eqnarray}
In particular,
\begin{eqnarray*}
G(\kk,1)-G(\kk,0)=\frac{1-\kk}{\kk}.
\end{eqnarray*}
Moreover, $G(\kk,z)=o(1)$ if $z\rightarrow\infty$ remains bounded away from $[1,\infty)$. More precisely: for $z\rightarrow\infty$ under the same condition, we have
\begin{eqnarray*}
G(\kk,z)\sim-\frac{1}{z}+\frac{\alpha}{z^2},
\text{ where }\alpha=G(\kk,0)-\frac{2-\kk}{1-\kk},\quad
\frac{\partial^{s}}{\partial z^{s}}\,G(\kk,z)=O\Big{(}z^{-(s+1)}\Big{)}.
\end{eqnarray*}
\end{prop}
\begin{proof}
First, we note the identity
\begin{eqnarray*}
\int\limits_{0}^{1}f(x)\d\mu_{\kk}(x)=\sum\limits_{\ell=2}^{\infty}(1-\kk)\kk^{\ell-2}\int\limits_{0}^{1}f\Big{(}\frac{1}{\ell-x}\Big{)}\d \mu_{\kk}(x),
\end{eqnarray*}
provided that all integrals are absolutely convergent. This follows from Proposition \ref{prop1}, the Property 3. In the special case, for $f(x)=(\frac{1}{x}-z)^{-1}$, this reduces to
\begin{eqnarray*}
G(\kk,z)=\int\limits_{0}^{1}\frac{1}{\frac{1}{x}-z}\d\mu_{\kk}(x)=
\sum\limits_{\ell=2}^{\infty}(1-\kk)\kk^{\ell-2}\int\limits_{0}^{1}\frac{1}{\ell-x-z}\d \mu_{\kk}(x).
\end{eqnarray*}
Thus,
\begin{eqnarray*}
G(\kk,z+1)-\kk G(\kk,z)=(1-\kk)\int\limits_{0}^{1}\frac{1}{1-x-z}\d\mu_{\kk}(x).
\end{eqnarray*}
Now, let us use the identity
\begin{eqnarray*}
\frac{1}{1-x-z}=\frac{1}{(1-z)^2}\cdot\frac{1}{\frac{1}{x}-\frac{1}{1-z}}+\frac{1}{1-z}.
\end{eqnarray*}
This gives the functional equation (\ref{centr}). The regularity property is immediate. \end{proof}

Using the same method as in \cite{alkauskas2} we see that
\begin{eqnarray}
G(\kk,z+1)=
(1-\kk)\sum\limits_{a,b,c,d\geq 0, \atop ad-bc=1}
\frac{\kk^{\imath(\frac{a+c}{b+d})}(1-\kk)^{\jmath(\frac{a+c}{b+d})}}{\big{[}(a+c)z-(b+d)\big{]}(cz-d)};
\label{expan}
\end{eqnarray}
here $\imath$ and $\jmath$ stand for the number of maps $T$ and $R$ (see the next section), respectively,  needed to obtain the rational number $\frac{a+c}{b+d}$ from the root $\frac{1}{1}$ in the Calkin-Wilf tree \cite{calkin-wilf}. So, $G(\kk,z)$ is holomorphic in both variables. \\

%Now we will prove the symmetry property.
%\begin{prop}The function $G(\kk,z)$ in the cut plane satisfies the symmetry property.
%\begin{eqnarray*}
%G(1-\kk,z+1)=-\frac{1}{z^2}G\Big{(}\kk,\frac{1}{z}+1\Big{)}-\frac{1}{z}.
%\end{eqnarray*}
%\end{prop}
%\begin{proof}Assume that
%\begin{eqnarray*}
%\frac{a+c}{b+d}=R^{\imath_{1}}T^{\jmath_{1}}\cdots R^{\imath_{k}}T^{\jmath_{k}}\Big{(}\frac{1}%{1}\Big{)},\quad
%\imath_{s},\jmath_{s}\in\{0,1\},\quad \imath_{s}+\jmath_{s}=1,\quad 1\leq s\leq k,
%\end{eqnarray*}
%is the way to obtain the rational number $\frac{a+c}{b+d}$ in the Calkin-Wilf tree. Then, if we %interchange $T$ and $R$ we get
%\begin{eqnarray*}
%\frac{b+d}{a+c}=T^{\imath_{1}}R^{\jmath_{1}}\cdots T^{\imath_{k}}R^{\jmath_{k}}\Big{(}\frac{1}%{1}\Big{)}.
%\end{eqnarray*}
%So,
%\begin{eqnarray*}
%G(1-\kk,z+1)=
%\kk\sum\limits_{a,b,c,d\geq 0, \atop ad-bc=1}
%\frac{(1-\kk)^{\imath(\frac{a+c}{b+d})}\kk^{\jmath(\frac{a+c}{b+d})}}{\big{[}(a+c)z-(b+d)\big{]}%(cz-d)}\\
%=\kk\sum\limits_{a,b,c,d\geq 0, \atop ad-bc=1}
%\frac{\kk^{\imath(\frac{b+d}{a+c})}(1-\kk)^{\jmath(\frac{b+d}{a+c})}}{\big{[}(a+c)z-(b+d)\big{]}%(cz-d)}
%\end{eqnarray*}
%\end{proof}
Let $\mathcal{D}=\{\kk\in\mathbb{C}:|\kk|< 1,|1-\kk|\leq 1\}$. This is the definition domain of the function $G(\kk,z)$ in variable $\kk$. If $\kk\in\mathcal{D}$ and $\kk\rightarrow 0_{+}$, then the function $\mu_{\kk}$ tends pointwise to the function which is $0$ in $[0,1)$ and $1$ at $x=1$. Thus,
\begin{eqnarray*}
\lim\limits_{\kk\rightarrow 0_{+}}G(\kk,z)=G(0,z)=\frac{1}{1-z}.
\end{eqnarray*} 
This satisfies the functional equation (\ref{centr}) in case $\kk=0$. On the other hand, if $\kk\rightarrow 1_{-}$, then the function $\mu_{\kk}$ tends pointwise to the function which is $0$ at $x=0$ and $1$ in the interval $(0,1]$. Thus, we also get
\begin{eqnarray*}
\lim\limits_{\kk\rightarrow 1_{-}}G(\kk,z)=G(1,z)\equiv 0.
\end{eqnarray*} 

\section{Mean-modular forms}
Let $\mathfrak{h}$ be the upper half plane, and let $G_{2}(z)$ stands for the holomorphic quasi-modular Eisenstein series of weight $2$ \cite{serre}:
\begin{eqnarray*}
G_{2}(z)=\frac{\pi^{2}}{3}-8\pi^{2}\sum\limits_{n=1}^{\infty}\sigma_{1}(n)e^{2\pi i n z}.
\end{eqnarray*}
We will also use the standard normalization
\begin{eqnarray*}
E_{2}(z)=\frac{3}{\pi^2}\,G_{2}(z)=1-24e^{2\pi i z}-72e^{4\pi i z}-\cdots.
\end{eqnarray*}
Also, let $E_{4}$ and $E_{6}$, as usual, be normalized Eisenstein series of weights $4$ and $6$. We know that for $z\in\mathfrak{h}$,
\begin{eqnarray*}
G_{2}(z+1)=G_{2}(z),\quad G_{2}(-1/z)=z^{2}G_{2}(z)-2\pi i z. 
\end{eqnarray*}
There exist several extensions of the space $M_{k}$ of modular forms of weight $k$. One of the extensions is the space of the so called \emph{quasi-modular forms}, which are weight $k$ elements of the ring $\mathbb{C}[E_{2},E_{4},E_{6}]$. Now we describe another 
$?(x)$-related class of extension of $M_{k}$, and will later prove that these two extensions are isomorphic! This gives unexpected view on quasi-modular forms, where $?(x)$ essentially enters the picture.\\
 
We will need the following differential identities of Ramanujan \cite{zagier}:
\begin{eqnarray}
\frac{1}{2\pi i}E'_{2}=\frac{E^{2}_{2}-E_{4}}{12},\quad 
\frac{1}{2\pi i}E'_{4}=\frac{E_{2}E_{4}-E_{6}}{3},\quad
\frac{1}{2\pi i}E'_{6}=\frac{E_{2}E_{6}-E^{2}_{4}}{2}.
\label{ramanujan}
\end{eqnarray}

A direct calculation shows that $\frac{i}{2\pi}G_{2}(z)$ satisfies the functional equation (\ref{centr}) for $z\in\mathfrak{h}$. Let, as before, the number $\kk$ belong to $\mathcal{D}$. 
If $G(\kk,z)$ is the function from the previous subsection, then, if we set 
\begin{eqnarray*}
Q(\kk,z)=-\frac{6}{\pi i}\Big{(}G(\kk,z+1)-\frac{i}{2\pi}G_{2}(z)\Big{)}=-\frac{6}{\pi i}\,G(\kk,z+1)+E_{2}(z),
\end{eqnarray*}
we see that this function is uniformly bounded for $\Im(z)>\epsilon>0$, $|\kk|<1-\epsilon$, $|1-\kk|\leq 1$, and for $z\in\mathfrak{h}$, $\kk\in\mathcal{D}$, it satisfies the functional equation
\begin{eqnarray}
\setlength{\shadowsize}{2pt}\shadowbox{$\displaystyle{\quad
f(\kk,z)=\kk f(\kk,z-1)+\frac{1-\kk}{(1-z)^{k}}f\Big{(}\kk,\frac{z}{1-z}\Big{)}\quad}$}
\label{lyg}
\end{eqnarray}
for $k=2$. The function $Q$ is fundamental function which plays the same role among mean-modular forms (see below) as $E_{2}$ plays in the theory of quasi-modular forms.\\

 Let $U,S,I,T,R$ be the standard $2\times 2$ matrixes:
\begin{eqnarray*}
U=\left(\begin{array}{cc}0 & 1 \\-1 & 1 \\ \end{array}\right),\quad
S=\left(\begin{array}{cc}0 & 1 \\-1 & 0 \\ \end{array}\right),\quad
I=\left(\begin{array}{cc}1 & 0 \\0 & 1 \\ \end{array}\right),\quad
T=\left(\begin{array}{cc}1 & 1 \\0 & 1 \\ \end{array}\right),\quad
R=\left(\begin{array}{cc}1 & 0 \\1 & 1 \\ \end{array}\right).
\end{eqnarray*}
The matrices $U,S$ satisfy $U^3=S^2=I$, and freely generate the modular group,  while $T=U^2S$, $R=US$ (all relations are considered modulo $\pm I$). Our main interest is the equation (\ref{lyg}) in case $\kk=\frac{1}{2}$, $k=2$, since this, as we have seen, is directly related to the Minkowski question mark function. Nevertheless, suppose $f(\kk,z)$ satisfies (\ref{lyg}), and let us consider this identity as the one for the function of two complex variables $\kk$ and $z$. At the one end of the real interval $\kk\in[0,1]$ (of course, $1\notin\mathcal{D}$, but suppose we are allowed to plug it), the function $f(\kk,z)$ is $T$-periodic:
\begin{eqnarray*}
f(1,z)=f(1,z)|T^{n},\quad n\in\mathbb{Z}. 
\end{eqnarray*}
The $R$-periodicity holds at the other end:
\begin{eqnarray*}
f(0,z)=f(0,z)|R^{n},\quad n\in\mathbb{Z}. 
\end{eqnarray*}
The two matrices $T$ and $R$ generate the whole modular group, and $R$ and $T$ are primitive elements there (i.e. not powers of other matrices) of infinite order. For example,
\begin{eqnarray*}
Q(1,z)=E_{2}(z),\quad 
Q(0,z)=\frac{6}{\pi i z}+E_{2}(z),
\end{eqnarray*}
which are, respectively, $T$- and $R$-periodic. So, generally, the function $f(\kk,z)$ cannot be called a modular form, but it rather as if interpolates a modular form, and  we think that the name mean-modular form (that is, a modular form on average), is apt. \\

Consider holomorphic functions $\ell(\kk)$, defined for $\kk\in\mathrm{int}(\mathcal{D})$ (interior). The set of these functions is then a ring, which we denote by $\f$ (thus, zeros of such functions can accumulate only at the boundary). 
\begin{defin}
Let $k\in2\mathbb{N}$. The function $f(\kk,z)$ is called \emph{a weight $k$ mean-modular form}, or \emph{MMF}, if 
\begin{itemize}
\item[i)] 
it is bivariate holomorphic function and satisfies the functional equation (\ref{lyg}) for $z\in\mathfrak{h}$, $\kk\in\mathcal{D}$;
\item[ii)]for every $\epsilon>0$ there exist a constant $C(\epsilon)$ such that $|f(\kk,z)|<C(\epsilon)$ for $\Im(z)>\epsilon$, $|\kk|<1-\epsilon$, $|1-\kk|\leq 1$.  
\end{itemize}
We denote the $\f-$module of mean-modular forms of weight $k$ by ${\sf Mmf}_{k}$. 
\end{defin}
In fact, there are many functions, constant in variable $\kk$, which satisfy the functional equation but fail the regularity condition. For example, when $k=2$ such functions are $j'(z)P(j(z))$, where $j(z)$  is the $j$-invariant, and $P$ is any polynomial.\\

%Let us denote the arc of the circle $|\kk|=\epsilon>0$, which is inside the disc $|1-\kk|\leq 1$, by %$\Omega(\epsilon)$. When $\epsilon\rightarrow0_{+}$, the length of the arc $\Omega(\epsilon)$ is %$\pi\epsilon+O(\epsilon^2)$. 
%\begin{defin}
%Let $Q\subset\mathfrak{h}$ be a compact set. The $j$-th coordinate of the mean-modular form $f(\kk,z)$ is %defined recurrently by
%\begin{eqnarray*}
%A_{j}(z)=\lim\limits_{\epsilon\rightarrow 0_{+}}\frac{1}{\pi i}\int\limits_{\Omega(\epsilon)}\frac{1}{\kk^{j+1}}
%\Big{(}f(\kk,z)-\sum\limits_{s=0}^{j-1}\kk^{s}A_{s}(z)\Big{)}\d\kk,\quad j\in\mathbb{N}_{0},\quad z\in Q,
%\end{eqnarray*}
%where the integral is taken via the arc $\Omega(\epsilon)$ from the bottom upwards, and empty sum is $0$ by %convention. The definition of $A_{j}(z)$ is extended to $\mathfrak{h}$ by expanding $Q$. Symbolically, we write
%\begin{eqnarray*}
%f(\kk,z)\sim\sum\limits_{j=0}^{\infty}\kk^{j}A_{j}(z).
%\end{eqnarray*}
%\end{defin} 
%Comparing the corresponding coefficients at powers of $\kk$ we obtain the basic relations among these functions:
%\begin{eqnarray}
%A_{j+1}|(I-US)=A_{j}|(I-SU^2S);
%\label{perein}
%\end{eqnarray}
%this holds for  $j\geq -1$, assuming $A_{-1}(z)\equiv 0$. \\

%\begin{thm}The space ${\sf Mmf}_{k}$ is finite dimensional:
%\begin{eqnarray*}
%\dim_{\f}({\sf Mmf}_{k})
%\end{eqnarray*}
%\end{thm}
\section{Mean-modular sections}
\begin{defin}
We call a function $T(z)$ \emph{a mean-modular section}, or \emph{MMS}, of weight $k$, if there exists a mean-modular form $f(\kk,z)$ of weight $k$ such that
\begin{eqnarray*}
T(z)=f\Big{(}\frac{1}{2},z\Big{)}.
\end{eqnarray*}
\end{defin}
Denote the $\mathbb{C}-$linear space of MMS of weight $k$ by ${\sf Mms}_{k}$. So, $\mathrm{dim}_{\mathbb{C}}({\sf Mms}_{k})=\mathrm{dim}_{\f}({\sf Mmf}_{k})$. The main motivation of this paper is thus the following facts:
\begin{itemize}
\item[$\lozenge$] if $T(z)$ is a modular form for ${\sf PSL}_{2}(\mathbb{Z})$, then $T(z)$ is a MMS of the same weight.
\item[$\lozenge$] ``Sporadic"  solutions of the three term functional equation (\ref{lyg}) (for a specific $\kk$), which are also in $M_{k}(\Gamma(N))$, do not qualify MMS (see Section \ref{sporadic}).
\item[$\lozenge$] Most importantly,
\begin{eqnarray*}
\int\limits_{0}^{1}\frac{x}{1-x(z+1)}\d ?(x)-\frac{i}{2\pi}G_{2}(z).
\end{eqnarray*}
is a MMS of weight $2$.
\end{itemize}
Our first main result of this paper shows that the space ${\sf Mms}_{k}$ possess the same property which implies ``the unreasonable effectiveness of modular forms", acccording to Don Zagier. Let $ \widetilde{M_{k}}$ be the space of quasi-modular forms of weight $k$.
\begin{thm}
\label{thm1}
The linear space ${\sf Mms}_{k}$ is finite dimensional, and there exists the canonical isomorphism
\begin{eqnarray*}
\psi:{\sf Mms}_{k}\mapsto \widetilde{M_{k}},
\end{eqnarray*}
which is given by
\begin{eqnarray*}
\psi(T(z))=\psi\Big{(}f\Big{(}\frac{1}{2},z\Big{)}\Big{)}=f(1,z),
\end{eqnarray*}
where $f$ is a mean-modular form giving rise to a mean modular section $T(z)$.
\end{thm}
Thus, for example, $\psi(Q)=E_{2}$.
\section{Homomorphisms of MMF}
\subsection{Serre's derivative}Similarly as in (\cite{zagier}, Section 5.1), we prove the following
\begin{prop}If $f$ is a $MMF$ of weight $k$, then
\begin{eqnarray*}
\vartheta_{k}(f)=\frac{1}{2\pi i}\frac{\partial}{\partial z}f(\kk,z)-\frac{k}{12}\cdot E_{2}(z)\cdot f(\kk,z)
\end{eqnarray*}
is a MMF of weight $k+2$. 
\end{prop}
If we apply the operator $\vartheta_{k}$ to weight $k$ MMF and it is unambiguous, we may drop the subscript $k$.
\begin{proof}Let, for simplicity, $f'=\frac{\partial}{\partial z}f$, and we ommit the first variable $\kk$. Then, if $f$ is a MMF, then
\begin{eqnarray}
f'(z)-\kk f'(z-1)-\frac{1-\kk}{(1-z)^{k+2}}f'\Big{(}\frac{z}{1-z}\Big{)}
=\frac{(1-\kk)k}{(1-z)^{k+1}}f\Big{(}\frac{z}{1-z}\Big{)}.
\label{der}
\end{eqnarray}
Further, let $u(z)=f(z)E_{2}(z)$. Then, according to the properties of $G_{2}(z)$, we have:
\begin{eqnarray*}
u(z)&=&f(z)E_{2}(z),\\
\kk u(z-1)&=&\kk f(z-1)E_{2}(z),\\
\frac{1-\kk}{(1-z)^{k+2}}u\Big{(}\frac{z}{1-z}\Big{)}&=&\frac{1-\kk}{(1-z)^k}f\Big{(}\frac{z}{1-z}\Big{)}E_{2}(z)+\frac{6i(1-\kk)}{\pi (1-z)^{k+1}}f\Big{(}\frac{z}{1-z}\Big{)}.
\end{eqnarray*}
Thus,
\begin{eqnarray}
u(z)-\kk u(z-1)-\frac{1-\kk}{(1-z)^{k+2}}u\Big{(}\frac{z}{1-z}\Big{)}
=-\frac{(1-\kk)6i}{\pi(1-z)^{k+1}}f\Big{(}\frac{z}{1-z}\Big{)}.
\label{derr}
\end{eqnarray}
Comparing (\ref{der}) and (\ref{derr}), we get the needed property that $\frac{1}{2\pi i }f'-\frac{k}{12}E_{2}f$ is a MMF. The second assertion of the proposition is obvious.
\end{proof}
We can see indeed that the operator $\vartheta$ is indeed a ``derivation". Let $f$ be weight $k$ mean modular form, and $g$ be weight $\ell$ modular form. Then $fg$ is weight $k+\ell$ MMF, and
\begin{eqnarray}
\vartheta_{k+\ell}(fg)=\frac{1}{2\pi i}(f'g+fg')-\frac{k+\ell}{12}E_{2}fg=\vartheta_{k}(f)g+f\vartheta_{\ell}(g);
\label{derivation}
\end{eqnarray}
So, $\vartheta$ satisifes the Leibniz rule.
%\subsection{Rankin-Cohen brackets}
\section{Isomorphism between mean-modular forms and quasi-modular forms} 
Let, as before, $M_{k}$ and $\widetilde{M}_{k}$ stand for the $\mathbb{C}-$linear space of weight $k$ modular and quasi-modular forms, respectively. We will define another linear operator from the space of MMF of weight $k$ to scalar extension of $\widetilde{M}_{k}$.
\begin{defin}Let us define
\begin{eqnarray*}
M^{\f}_{k}=M_{k}\otimes_{\mathbb{C}}\f,\quad \widetilde{M^{\f}_{k}}=\widetilde{M}_{k}
\otimes_{\mathbb{C}}\f.
\end{eqnarray*}
\end{defin}
So, these two are $\f-$ modules of weight $k$ modular and quasi-modular forms, respectively, only the ``constants" are changed from the field $\mathbb{C}$ to the ring $\f$. So,
\begin{eqnarray*}
M^{\f}_{k}\subset{\sf Mmf}_{k}.
\end{eqnarray*}
Now, define the linear map
\begin{eqnarray*}
\E:{\sf Mmf}_{k}\mapsto\widetilde{M^{\f}_{k}}
\end{eqnarray*}
as follows. If $f(\kk,z)$ is a MMF, then put
\begin{eqnarray*}
\E f(\kk,z)=\lim\limits_{n\in\mathbb{Z}, n\rightarrow\infty}f(\kk,z+n).
\end{eqnarray*}
Thus, if $f\in M^{\f}_{k}$, then $\E f=f$.
Note that it really does not matter that we work over $\f$ - our chief interest is the space ${\sf Mms}_{k}$ and so the special case $\kk=\frac{1}{2}$, and this will reduce the ring of constants back to $\mathbb{C}$.
For example, based on Proposition \ref{prop2}, we have
\begin{eqnarray*}
\E{Q}&=&E_{2};\\
\E\vartheta(Q)&=&-\frac{1}{12}E^{2}_{2}-\frac{1}{12}E_{4};\\
\mathfrak{E}{\vartheta^{2}(Q)}&=&\frac{1}{72}E^{3}_{2}+\frac{1}{72}E_{2}E_{4}+\frac{1}{36}E_{6}.
\end{eqnarray*}
Let us define
\begin{eqnarray*}
{\sf Mmf}=\mathop{\oplus}_{k\in 2\mathbb{N}_{0}}{\sf Mmf}_{k}.
\end{eqnarray*}
Our basic result of this paper, which also implies Theorem \ref{thm1}, reads as follows. 
\begin{thm}
\label{thm2}
The following holds.
\begin{itemize}
\item[i)]
The map $\E$ is an isomorphism of $\f-$modules. This provides the product structure inside MMF by the following construction. If $f\in{\sf Mmf}_{k}$, $g\in{\sf Mmf}_{\ell}$, then we define $f\star g\in{\sf Mmf}_{k+\ell}$ by
\begin{eqnarray*}
f\star g=\E^{-1}\big{(}\E(f)\cdot\E(g)\big{)}.
\end{eqnarray*} 
The product $``\star"$ turns ${\sf Mmf}$ into the graded algebra. If $g\in M^{\f}_{k}$, then
\begin{eqnarray*}
f\star g=f\cdot g.
\end{eqnarray*}
\item[ii)]The map $\E$ commutes with the derivation $\vartheta$. That is, let $f\in{\sf Mmf}_{k}$. Then
\begin{eqnarray*}
\E\vartheta(f)=\vartheta(\E f),
\end{eqnarray*}
where the second $\vartheta$ is the map $\widetilde{M^{\f}_{k}}\mapsto\widetilde{M^{\f}_{k+2}}$.
\item[iii)]The map $\vartheta$ and the product $\star$ are compatable with the Leibniz rule; that is, if $f\in{\sf Mmf}_{k}$, $g\in{\sf Mmf}_{\ell}$, then
\begin{eqnarray*}
\vartheta(f\star g)=\vartheta(f)\star g+f\star\vartheta(g).
\end{eqnarray*}
\end{itemize}
\end{thm}
\begin{proof}The part ii) is proved by a direct check. We have already verified iii) in case $g\in M^{\f}_{k}$. Let $f\in{\sf Mmf}_{k}$, $g\in{\sf Mmf}_{\ell}$. Then we have:
\begin{eqnarray*}
\vartheta(f\star g)&\mathop{=}\limits^{i)}&\vartheta\Big{(}\E^{-1}\big{(}\E f\cdot\E g\big{)}\Big{)}\mathop{=}\limits^{ii)}
\E^{-1}\Big{(}\vartheta\big{(}\E f \cdot\E g \big{)}\Big{)}\\
&\mathop{=}\limits^{(\ref{derivation})}&\E^{-1}\Big{(}\vartheta(\E f)\cdot\E g+\E f\cdot\vartheta(\E g)\Big{)}\\
&\mathop{=}\limits^{ii)}&\E^{-1}\Big{(}\E\vartheta(f)\cdot\E g\Big{)}+\E^{-1}\Big{(}\E f\cdot\vartheta(\E g)\Big{)}\\
&\mathop{=}\limits^{i)}&\vartheta(f)\star g+f\star\vartheta(g).
\end{eqnarray*} 
This proves iii).
\end{proof}
As a warming up, suppose that we have already proven that
\begin{eqnarray*}
\E\vartheta^{s}(g)=q_{s}E^{s+1}_{2}+\{\text{terms involving }E_{2}\text{ and at least one of }E_{4},E_{6}\text{ of total weight }2s+2\}.
\end{eqnarray*}
Then from (\ref{ramanujan}) and the definition of $\vartheta$ we derive that
\begin{eqnarray}
q_{0}=1,\quad q_{s+1}=-\frac{s+1}{12}q_{s}\Longrightarrow q_{s}\neq 0.
\label{qs}
\end{eqnarray}
The $\E$ operation will help us to rule out linear dependence among the set of MMF of weight $k$, which we will now construct. The crucial ingrediant is a well-known and important fact which claims that the Eisenstein series $E_{2}, E_{4}$ and $E_{6}$ are algebraically independent \cite{zagier}.

\subsection{Building MMF} 
So, we will now list the elements of ${\sf Mmf}_{k}$ whch we already know. For simplicity, we ommit the first variable $\kk$.\\

\begin{itemize}

\item[Weight] $2$, $\mathrm{dim}_{\f}({\sf Mmf}_{2})=1$: $Q$.

\item[Weight] $4$, $\mathrm{dim}=2$: $\vartheta(Q)$, $E_{4}$. These are linearly independent, since the first MMF is non-periodic, while the second is. 

\item[Weight] $6$, $\mathrm{dim}=3$: $\vartheta^{2}(Q)$, $E_{4} Q$, $E_{6}$. These three MMF are also linearly independent. Indeed suppose the contrary,
\begin{eqnarray*}
a\vartheta^{2}(Q)+bE_{4}Q+cE_{6}=0.
\end{eqnarray*}
Substitute $z\mapsto z+n$ and take now the limit $n\rightarrow\infty$. Thus, in fact we are applying the $\E$ operator:
\begin{eqnarray}
a\E\vartheta^{2}(Q)+bE_{4}\E Q+cE_{6}=0.
\label{weight6}
\end{eqnarray}
This is the combination of the products of $E_{2}$, $E_{4}$ and $E_{6}$. The coefficient at $E_{2}^{3}$ is $aq_{2}$, so the algebraic independence of $E_{2}$, $E_{4}$ and $E_{6}$ implies that $a=0$. Then the coefficient at $E_{2}E_{4}$ of the remaining terms in (\ref{weight6}) is $bq_{0}$, and this again yields $b=0$. We therefore find that $a=b=c=0$. Essentially the same method works to show that for every weight, the below constructed MMF are linearly independent. We can also calculate
\begin{eqnarray*}
\E Q\cdot\E\vartheta(Q)=-\frac{1}{12}E^{3}_{2}-\frac{1}{12}E_{2}E_{4}=-6\E\vartheta^{2}(Q)+\frac{1}{6}\E E_{6}.
\end{eqnarray*}
So, 
\begin{eqnarray*}
Q\star\vartheta(Q)=-6\vartheta^{2}(Q)+\frac{1}{6}E_{6}.
\end{eqnarray*}

\item[Weight] $8$, $\mathrm{dim}=4$:  $\vartheta^{3}(Q)$, $E_{4}\vartheta(Q)$, $E_{6}Q$, $E_{8}$. Note that, for example, $\vartheta(E_{4} Q)$ does not give anything new, since using the properties (\ref{ramanujan}) and (\ref{derivation}), we have:
\begin{eqnarray*}
\vartheta(E_{4} Q)=\vartheta(E_{4})Q+E_{4}\vartheta(Q)=
-\frac{1}{3}E_{6}Q+E_{4}\vartheta(Q).
\end{eqnarray*}

\item[Weight] $10$, $\mathrm{dim}=5$: $\vartheta^{4}(Q)$, $E_{4}\vartheta^{2}(Q)$, $E_{6}\vartheta(Q)$, $E_{8}Q$, $E_{10}$.\\

\item[Weight] $12$, $\mathrm{dim}=7$: $\vartheta^{5}(Q)$, $E_{4}\vartheta^{3}(Q)$, $E_{6}\vartheta^{2}(Q)$, 
$E_{8}\vartheta(Q)$,  $E_{10}Q$, $E_{12}$, $\Delta$.\\

\item[Weight] $14$, $\mathrm{dim}=8$: $\vartheta^{6}(Q)$, $E_{4}\vartheta^{4}(Q)$, $E_{6}\vartheta^{3}(Q)$, $E_{8}\vartheta^{2}(Q)$, $E_{10}\vartheta(Q)$, $E_{12}Q$, $\Delta Q$, $E_{14}$.\\

\item[Weight] $16$, $\mathrm{dim}=10$: $\vartheta^{7}(Q)$, $E_{4}\vartheta^{5}(Q)$, $E_{6}\vartheta^{4}(Q)$, $E_{8}\vartheta^{3}(Q)$, $E_{10}\vartheta^{2}(Q)$, $E_{12}\vartheta(Q)$, $\Delta\vartheta(Q)$, $E_{14}Q$, $E_{16}$, $\Delta E_{4}$. These $10$ MMF are linearly independent, which is proven by the same method. Indeed, we take the ``$\E$" operator of the linear dependeancy of the above $10$ MMF. That all coefficients vanish, we prove by inspecting first the coeffcient at $E^{8}_{2}$, then at $E_{4}E^{6}_{2}$, and so on. In fact, there are already two terms which contain $E^{2}_{2}$; these are $E_{12}\E\vartheta(Q)$, $\Delta\E\vartheta(Q)$. But $E_{12}$ and $\Delta$ are linearly independent.

\end{itemize}
So, we see that, for even $k\geq 2$,
\begin{eqnarray*}
\mathrm{dim}_{\f}({\sf Mmf}_{k})=
\sum\limits_{\ell=0}^{k/2}\mathrm{dim}_{\mathbb{C}}(M_{k-2\ell})+1.
\end{eqnarray*}

%Recall that $E_{k}(i)=0$ for $4\nmid k$ and $E_{k}(\rho)=0$ for %$6\nmid k$, $\rho=\frac{1+\sqrt{3}}{2}$. 
\section{Modular solutions}
\label{sporadic}
We will now show that the requirement that a mean-modular form is holomorphic in variable $\kk$ is essential and strong, since there exists too many functions which satisfy (\ref{lyg}) for certain particular fixed $\kk$. Moreover, such functions can even be modular forms for congruence subgroups. Consequently, such ``sporadic" solutions do not qualify as MMF.\\

Let $N\in\mathbb{N}$, $k\in 2\mathbb{N}$. Consider the space of modular forms $M_{k}(\Gamma(N))$. Let $\mathbf{u}(z)=(u_{1}(z),u_{2}(z),\ldots, u_{\ell}(z))$ be the basis of this space. We know that for any $u(z)\in M_{k}(\Gamma(N))$, both $u(z-1)$ and $(1-z)^{-k}u(z/(1-z))$ belong to  $M_{k}(\Gamma(N))$. This simply follows from the fact that $\Gamma(N)$ is a normal subgroup of $\Gamma(1)$. So there exists two matrices $A$ and $B$ such that
\begin{eqnarray*}
\mathbf{u}(z-1)^{T}=A\mathbf{u}(z)^{T},\quad 
(1-z)^{-k}\mathbf{u}(z/(1-z))^{T}=B\mathbf{u}(z)^{T}.
\end{eqnarray*}
We want a function
\begin{eqnarray*}
\sum\limits_{i=1}^{\ell}a_{i}u_{i}(z),\quad a_{i}\in\mathbb{C},
\end{eqnarray*}
to satisfy (\ref{lyg}). There exists a non-zero vector $(a_{1},\ldots,a_{\ell})$  if and only if the determinant of the matrix $I-\kk A-(1-\kk)B$ vanishes:
\begin{eqnarray*}
P_{N,k}(\kk):=\det\big{(}I-B+\kk(B-A)\big{)}=0.
\end{eqnarray*}
So, each pair $N\geq 2$, $k\in 2\mathbb{N}$, generates the polynomial $P_{N,k}(\kk)$, and each root of this polynomial produces the element of $M_{k}(\Gamma(N))$ that also satisfies (\ref{lyg}). 
For example, let $N=2$, $k=2$. The space $M_{2}(\Gamma(2))$ is $2-$dimensional and is spanned by $\vartheta^{4}(0,1/2;z)$ and $\vartheta^{4}(1/2,0;z)$, the Jacobi's theta functions (see further). The polynomial $P_{2,2}(\kk)=3\kk(1-\kk)$. So, in this case only $\kk=0$ belongs to $\mathcal{D}$. Anyway, using approach via theta constants, we have calculated many possible $\kk$, and there are plenty of whose which belong to $\mathcal{D}$; for example, $\kk=\frac{1}{2}+\frac{1}{2}i$ is one of them. The approach via theta consists consists of the following. \\

Let us define, for $a,b\in\mathbb{R}$, $k\in\mathbb{N}$ (no relation to the weight!), $z\in\mathfrak{h}$, the theta-constants \cite{dolg,mumford}
\begin{eqnarray*}
\vartheta(a,b;z)_{k}&=&\sum\limits_{n\in\mathbb{Z}}e^{k\pi i[(a+n)^{2}z+2b(a+n)]}=\vartheta(a,kb;kz)_{1};\\
\vartheta(a,b;z)'_{k}&=&2k\pi i\sum\limits_{n\in\mathbb{Z}}(a+n)e^{k\pi i[(a+n)^{2}z+2b(a+n)]}=k\vartheta(a,kb;kz)'_{1}.
\label{theta}
\end{eqnarray*}
(No relation to the ``$\vartheta$" map!) The next identities are checked directly; they are either immediate, or follow from the Poisson summation formula.
\begin{prop}The functions $\vartheta(a,b;z)_{k}$ and $\vartheta(a,b;z)_{k}'$ for rational $a,b$ are modular forms of weights $1/2$ and $3/2$, respectively. Further, we have
\begin{itemize}
\item[1-1')]$\vartheta(a+1,b;z)_{k}=\vartheta(a,b;z)_{k}$;
\item[2-2')]$\vartheta(a,b+\frac{1}{k};z)_{k}=e^{2\pi ia}\vartheta(a,b;z)_{k}$;
\item[3-3')]$\vartheta(a,b;z+1)_{k}=e^{-k\pi i(a^{2}+a)}\vartheta(a,b+a+\frac{1}{2};z)_{k}$;
\item[4)]$\vartheta(-a,-b;z)_{k}=\vartheta(a,b;z)_{k}$;
\item[4')]$\vartheta(-a,-b;z)'_{k}=-\vartheta(a,b;z)'_{k}$;
\item[5)]$\vartheta(a,b;-\frac{1}{z})_{k}=k^{-1/2}(-iz)^{1/2}e^{2k\pi i ab}\sum\limits_{s=0}^{k-1}\vartheta(b+\frac{s}{k},-a;z)_{k}$.
\item[5')]$\vartheta(a,b;-\frac{1}{z})'_{k}=k^{-1/2}i(-iz)^{3/2}e^{2k\pi i ab}\sum\limits_{s=0}^{k-1}\vartheta(b+\frac{s}{k},-a;z)'_{k}$.
\end{itemize}
1-1', 2-2' and 3-3' mean that the same transformation rules hold for $\vartheta(a,b;z)_{k}$ and $\vartheta(a,b;z)'_{k}$.
\end{prop}

So, we start from any product of these theta constants, that include only rational parameters $a,b$, and which amount to the total weight of, say, $2$. This function satisfies transformation properties under  $z\mapsto z+1$, $z\mapsto -z^{-1}$. It belongs to the finite orbit, and thus this also reduces to the condition for the determinant. For example, let us consider the simplest case of weight $2$ and when these products are in fact $4$th powers of theta constants.

\subsection{Theta functions $\vartheta^{4}(a,b;z)_{1}$ for $4a,4b\in\mathbb{Z}$} There are three orbits in this case. First, the orbit-singleton $(\frac{1}{2},\frac{1}{2})$, which produce a zero theta constant. Further, the $3$-element orbit $(0,0)$, $(\frac{1}{2},0)$, $(0,\frac{1}{2})$, which was already investigated; these three functions are related via the Jacobi identity:
\begin{eqnarray*}
\vartheta^{4}(0,0;z)=\vartheta^{4}(1/2,0;z)+\vartheta^{4}(0,1/2;z).
\end{eqnarray*}
 The third orbit consists of $6$ elements $(0,\frac{1}{4})$, $(\frac{1}{4},0)$, $(\frac{1}{4},\frac{1}{4})$, $(\frac{3}{4},\frac{1}{4})$, $(\frac{1}{4},\frac{1}{2})$, and $(\frac{1}{2},\frac{1}{4})$. Therefore,
\begin{eqnarray*}
\mathbf{u}(z)^{T}=\left(\begin{array}{c}
\vartheta^4(0,1/4;z) \\
\vartheta^4(1/4,0;z)\\
\vartheta^4(1/4,1/4;z)\\
\vartheta^4(3/4,1/4;z)\\
\vartheta^4(1/4,1/2;z)\\
\vartheta^4(1/2,1/4;z)\\
\end{array}\right),
\end{eqnarray*}
and the space generated by all six components is invariant under the action of  $T$ and $S$.
\subsection{Theta functions $\vartheta^{4}(a,b;z)_{1}$ for $(6a,6b)\in\mathbb{Z}^2$, $(2a,2b)\notin\mathbb{Z}^2$} In this case the theta functions split into three orbits: $Q_{1}$, consisting of $4$ functions with rational pairs $(a,b)=(\frac{1}{6},\frac{1}{6})$, $(\frac{5}{6},\frac{1}{6})$, $(\frac{1}{6},\frac{1}{2})$, $(\frac{1}{2},\frac{1}{6})$; $Q_{2}$, consisting of $4$ rational pairs $(\frac{1}{3},\frac{1}{3})$, $(\frac{1}{3},\frac{1}{2})$, $(\frac{1}{2},\frac{1}{3})$, $(\frac{2}{3},\frac{1}{3})$, and $Q_{3}$, consisting of $8$ pairs $(0,\frac{1}{6})$, $(0,\frac{1}{3})$, $(\frac{1}{6},0)$, $(\frac{1}{6},\frac{1}{3})$, $(\frac{1}{3},0)$, $(\frac{1}{3},\frac{1}{6})$, $(\frac{2}{3},\frac{1}{6})$, $(\frac{5}{6},\frac{1}{3})$. For example, 
\begin{eqnarray*}
\mathbf{u}(z)^{T}=\left(\begin{array}{c}
\vartheta^4(1/6,1/6;z) \\
\vartheta^4(5/6,1/6;z)\\
\vartheta^4(1/6,1/2;z)\\
\vartheta^4(1/2,1/6;z)\\
\end{array}\right),
\end{eqnarray*}
 and the space generated by all four components is invariant under the action of  $T$ and $S$; this is the subspace of $M_{2}(\Gamma_{2}(18))$. In fact, we can use not only the fourth powers but products of different theta constants, this produces the plethora of solutions to (\ref{lyg}) with many different algebraic $\kk$.

\end{document}